\documentclass{amsart}
\usepackage{amsmath,dsfont,amsfonts}
\usepackage{amssymb}
\usepackage{amsthm}
\usepackage{verbatim}
\usepackage{color}
\usepackage{tikz}

\newtheorem{teo}{Theorem}[section]
\newtheorem{prop}[teo]{Proposition}
\newtheorem{lm}[teo]{Lemma}

\newtheorem{rem}[teo]{Remark}

\newcommand{\RR}{{\mathbb{R}}}
\newcommand{\NN}{{\mathbb N}}
\newcommand{\ds}{\displaystyle}
\newcommand{\der}{\partial}
\newcommand{\ep}{\varepsilon}

\newcommand{\om}{\Omega}
\newcommand{\dive}{\text{\normalfont div}}

\def\qed{\hfill$\square$\vspace{0.5cm}}    

\numberwithin{equation}{section}

\begin{document}

\title[A transmission problem on a polygonal partition]{A transmission problem on a polygonal partition: regularity and shape differentiability}

\author[E.~Beretta et al.]{Elena~Beretta}
\address{ Dipartimento di Matematica ``Brioschi'',
Politecnico di Milano, Italy}
\email{elena.beretta@polimi.it}
\author[]{Elisa~Francini}
\address{Dipartimento di Matematica e Informatica ``U. Dini'',
Universit\`{a} di Firenze, Italy}
\email{elisa.francini@unifi.it}
 \author[]{Sergio~Vessella}
 \address{Dipartimento di Matematica e Informatica ``U. Dini'',
Universit\`{a} di Firenze, Italy}
\email{sergio.vessella@unifi.it}

\date{\today}

\keywords{polygonal inclusions, conductivity equation, shape
derivative}

\subjclass[2010]{35R30, 35J25, 49Q10, 49N60}
\begin{abstract}
We consider a transmission problem on a polygonal partition for the two-dimensional conductivity equation. For suitable classes of partitions we establish the exact behaviour of the gradient of solutions in a neighbourhood of the vertexes of the partition. This allows to prove shape differentiability of solutions and to establish an explicit formula for the shape derivative.

\end{abstract}

\maketitle


\section{Introduction}
In this paper we consider the conductivity equation in a bounded planar domain 
\begin{equation}\label{condeq}
 \dive(\sigma\nabla u)  =  0\mbox{ in }\Omega\subset\mathbb{R}^2.
\end{equation}
We assume the conductivity $\sigma$ of the form
\begin{equation}\label{cond}
\sigma=\sum_{i=1}^M\sigma_j\chi_{\mathcal{P}_i},
\end{equation}
where $\mathcal{P}=\{\mathcal{P}_i\}_{i=1}^M$ is a polygonal regular partition of the background medium $\Omega$. 

This assumption on the conductivity is rather natural  and arises, for example, in applications to geophysics, medical imaging and nondestructive testing of materials where the medium under investigation contains regions with different conducting properties. 
Moreover, piecewise constant coefficients represent a class of unknown functions in which Lipschitz stable reconstruction from boundary data can be expected (see \cite{AV}, \cite{BF}, \cite{GS}, for example) and it appears in many finite-element scheme used for effective reconstruction.\\

Our main goal is to study the differentiability properties of solutions to the conductivity equation (\ref{condeq}) with respect to movements of the partition $\mathcal{P}$ i.e. to establish the existence of the shape derivative of $u$. \\

This analysis is motivated by the study of the inverse conductivity problem of recovering $\sigma$ of the form (\ref{cond}) from boundary measurements.  
More precisely, in order to derive quantitative Lipschitz stability estimates for a conductivity parameter, satisfying (\ref{cond}), in terms of the Neumann to Dirichlet map $\mathcal{N}_{\sigma}$, a crucial role is played by the differentiability properties of the map \[
F:\sigma\rightarrow \mathcal{N}_{\sigma}\]
with respect to movements of the partition and by the knowledge of an explicit formula for the derivative. (see \cite{BdHFV} for the case of the Helmholtz equation).\\
In \cite{BFV17} we performed a first step proving differentiability of $F$ in the case of a single polygonal inclusion $\mathcal{P}$ contained in $\Omega$ and we derived rigorously for the first time an explicit formula for the shape derivative of $F$  expressed in terms of an integral on the boundaries  of the polygons in $\mathcal{P}$. 
One of the main issues in the study of shape differentiability is the regularity of the solution $u$ of the elliptic pde. The coefficients we consider have jumps on polygonal boundaries. The related solutions are H\"older continuous in the interior of the  domain  $\Omega$ (see \cite{deGiorgi} and \cite{Moser}) and  smooth (in fact analytic) in the interior of each polygon.  Across the sides of the polygons the solutions are continuous and have continuous conormal derivative (transmission conditions). Moreover, $\nabla u$ has a Lipschitz continuous extension from the interior of the polygon to the internal part of each side of the polygon (\cite{LN}).  When approaching the vertexes of the polygons the gradient becomes more singular and an analysis  of the exact behaviour of gradients of solutions in a neighbourhood of vertexes of $\mathcal{P}$ is needed.  In  the case of a single polygonal inclusion we used the analysis derived in \cite{BFI}.\\
 In the more general case considered in this paper the situation is far more complicated. In this case, again, a crucial step is played by the analysis of  the differentiability properties of the solutions in a neighbourhood of the points of intersection of the sides of the polygons but the behaviour of $u$ depends on how the sides of elements of the partition intersect at those points.\\  
In fact,  from \cite{PS}, it is known that for solutions of (\ref{condeq}) with conductivities $\sigma\in L^{\infty}(\Omega)$ satisfying 
\[
\lambda\leq \sigma\leq\Lambda \text{ a.e. in } \Omega\subset\mathbb{R}^2
\]
 the H\"older exponent $\alpha$ can be computed explicitly and has the form 
\[\alpha=\frac{4}{\pi}\arctan\left(\sqrt{\frac{\lambda}{\Lambda}}\right).\] 
This represents the worse H\"older exponent for solutions to (\ref{condeq})  and it is attained for solutions corresponding to partitions meeting in a vertex with four sides at a right angle. So, in general, the regularity of solutions to (\ref{condeq})  and (\ref{cond}) does not allow us to prove shape differentiability of $u$.\\ 
In this paper we succeed in determining classes of partitions for which the regularity of the solutions and its gradients at the points of intersection of the polygons is enough to guarantee differentiability of solutions $u$. Furthermore, we establish an explicit formula for the shape derivative of $u$, $u'$, on the boundary of $\Omega$.
The paper is organized as follows: in Section 2 we prove the estimate on the behaviour of $\nabla u$ in a neighbourhood of the points of the partition with no more than 3 sides intersecting. In Section 3 we use this estimate to prove the existence of the shape derivative $u'$ with respect to movements of the partition, to find and explicit representation formula on the boundary of $\Omega$  and derive some relevant consequences.

\section{Behaviour of  $\nabla u$ in a neighbourhood of a vertex of certain classes of partitions}\label{stima}
Let $B$ be the open disk of radius $r_0$ centered at the origin $O=(0,0)$ and let $\sigma$ be a piecewise constant coefficient defined in $\overline{B}$ expressed in polar coordinates by
\[
\sigma(\rho,\theta)=\left\{\begin{array}{rcl}
\sigma_1&\mbox{for }&\beta_0:=0\leq\theta<\beta_1,\\
\sigma_2&\mbox{for }&\beta_1\leq\theta<\beta_2,\\
\sigma_3&\mbox{for }&\beta_2\leq\theta<\beta_3:=2\pi,
\end{array}\right.
\]
where 
\[
0<\sigma_0\leq\sigma_k\leq \sigma_0^{-1}, \mbox{ for }k=1,2,3.
\]

Let $u\in H^1(B)$ be a solution to
\[
\dive(\sigma\nabla u)=0\mbox{ in }B.\]
For $k=1,2,3$, let us denote by
\[D_k=\{(\rho,\theta)\,:\, 0<\rho<r_0,\,\beta_{k-1}\leq\theta\leq\beta_k\}\]
and by
\[u_k=u_{|_{D_k}}.\]

Each function $u_k$ is harmonic in $D_k$ and transmission conditions at the boundaries of $D_k$ hold, that is, $u$ and $\sigma\frac{\der u}{\der n_k}$ are  continuous across these boundaries.
Moreover, by Theorem 1.1 in \cite{LN} each function $u_k$ can be extended as a $C^{1,\alpha}$ function up to the boundary of the sector $D_k$ and $C^{1,\alpha}$ norm of $u_k$ can be bounded in terms of the $L^2$ norm of $u$  uniformly on subsets of $\overline{D_k}$ that have positive distance from the origin. 
\begin{teo}\label{stimagrad}
If, for some $\overline{\beta}\in (0,\pi)$, 
\begin{equation}\label{ipangoli}\beta_{k}-\beta_{k-1}\leq \pi-\overline{\beta},\mbox{ for }k=1,2,3,\end{equation}
there exist $C>0$ and $\gamma>1/2$ depending only on $\overline{\beta}$, $r_0$ and $\sigma_0$, such that
\begin{equation}\label{gradest}\left|\nabla u_k(x,y)\right|\leq C 
\|u\|_{H^1(B)}dist((x,y),O)^{\gamma-1},\mbox{ for } (x,y)\in D_k.\end{equation}
\end{teo}

In order to prove Theorem \ref{stimagrad}, let us show the following expansion for solution $u$.
\begin{prop}\label{sviluppo} Under the same assumptions of Theorem \ref{stimagrad} the following expansion holds for $0<r\leq\frac{r_0}{2}$ and  $k=1,2,3$
\begin{equation}\label{formsvil}
	u_k(r,\theta)=u_k(0)+\sum_{j=1}^{\infty}r^{\gamma_j}\left(A_j^k\cos(\gamma_j\theta)+B_j^k\sin(\gamma_j\theta)\right)\mbox{ for }\theta\in (\beta_{k-1},\beta_{k}).
	 \end{equation}
The series are convergent uniformly in $0<r\leq\frac{r_0}{2}$ and their first derivatives are absolutely convergent in the same set.
The sequence $\gamma_j$ is monotone increasing, there are $c_1$ and $c_2$ such that
\begin{equation}\label{asintautoval}
	0<c_1\leq \frac{\gamma_j}{j}\leq c_2\mbox{ for all }j\in\NN,
\end{equation}
and
\begin{equation}
	\gamma_1>\frac{1}{2}.
\end{equation}
\end{prop} 
\begin{proof}
We follow the outline of \cite{BFI}. 
Let us define the function 
$a(\theta) = \sigma(r_0,\theta)$ for $\theta\in [0,2\pi]$ and introduce the weighted spaces $L^2_a(S^1), H^1_a(S^1)$ with norms
\begin{align*}
&\| v\|_{L^2_a(S^1)} = \left( \int_0^{2\pi}a(\theta)|v(\theta)|^2 d\theta \right)^{1/2},\\
&\| v\|_{H^1_a(S^1)} = \left( \int_0^{2\pi}a(\theta)\left(\left|\frac{\der v}{\der\theta}(\theta)\right|^2+ |v(\theta)|^2 \right)d\theta \right)^{1/2}.
\end{align*}
Define
\begin{equation}
\mathcal{L} v =\frac{1}{a}\frac{\partial}{\partial \theta}\left(a \frac{\partial}{\partial \theta}v\right).
\end{equation}
$\mathcal{L}$ is an unbounded, selfadjoint, positive elliptic operator with dense domain in $L^2_a(S^1)$, and $(\mathcal{L} +1)^{-1}$ is compact. Let us denote by $\gamma^2_j$, $(\gamma_j \geq 0)$ the positive eigenvalues of $\mathcal{L}$ that constitute its spectrum. We denote the corresponding complete orthonormal sequence by $\{ v^{(j)} \}$, which is a basis for $L^2_a(S^1)$.

The solution $u$ can be written, for $0<r<r_0$ as 
\begin{equation}\label{serieaut}u(r,\theta)=u(0)+\sum_{j=1}^{\infty} C_jr^{\gamma_j}v^{(j)}(\theta).\end{equation}
Since $u_r\in L^2_a(S^1)$ for $r=r_0$, we have
\begin{equation}\label{Kappa}K:=\sum_{j=1}^{\infty}C_j^2\gamma_j^2r_0^{2\gamma_j}<\infty.\end{equation}

The asymptotic behaviour of eigenvalues \eqref{asintautoval} is obtained  from the variational formulation for the eigenvalues: see, for example, \cite[Example 4.6.1]{D}.

We now want to estimate from below the first positive eigenvalue of $\mathcal{L}$.
Let $v\in H^1_a(S^1)$ be solution to
\begin{equation}\label{eqaut}
	\mathcal{L}v+\gamma^2v=0
\end{equation}
such that
\[\int_{0}^{2\pi}a v^2(\theta)d\theta=1\]
and $\gamma>0$.
The function $v(\theta)$ satisfies the equation
\[\frac{\der}{\der\theta}\left(a(\theta)\frac{\der}{\der\theta}v(\theta)\right)+\gamma^2a(\theta)v(\theta)=0,\mbox{ for }0\leq\theta\leq 2\pi,\]
with 
\[v(0)=v(2\pi).\]
Let $v_k=v_{|_{[\beta_{k-1},\beta_k]}}$ for $k=1,2,3$. By considering the equation in $[\beta_0,\beta_1]$ we have
\[v_1(\theta)=v_1(0)\cos(\gamma\theta)+\gamma^{-1}v_1^\prime(0)\sin(\gamma\theta).\]
By the transmission conditions at $\theta=\beta_1$ we get
\begin{align*}
&v_2(\beta_1)=v_1(\beta_1)=v_1(0)\cos(\gamma\beta_1)+\gamma^{-1}v_1^\prime(0)\sin(\gamma\beta_1),\\
&v_2^\prime(\beta_1)=\frac{\sigma_1}{\sigma_2}v_1^\prime(\beta_1)=\frac{\sigma_1}{\sigma_2}\left\{-v_1(0)\gamma\sin(\gamma\beta_1)+v_1^\prime(0)\cos(\gamma\beta_1)\right\}
\end{align*}
that can be written as
\[
\begin{pmatrix}v_2(\beta_1)\\v_2^\prime(\beta_1)\end{pmatrix}=M_1 \begin{pmatrix}v_1(0)\\v_1^\prime(0)\end{pmatrix}
\]
where
\[M_1=\begin{pmatrix}\cos\gamma(\beta_1-\beta_0)&\gamma^{-1}\sin\gamma(\beta_1-\beta_0)\\
-\frac{\sigma_1}{\sigma_2}\gamma \sin\gamma(\beta_1-\beta_0)
&\frac{\sigma_1}{\sigma_2}\cos\gamma(\beta_1-\beta_0)\end{pmatrix}.\]
In the same way, by writing explicitely the solution of the ordinary differential equation in $[\beta_1, \beta_2]$, exploiting the transmission conditions at $\theta=\beta_2$,  considering the solution in $[\beta_2,\beta_3]$ and, finally, using the transmission conditions at $\theta=\beta_3=2\pi$ we get
\[\begin{pmatrix}v_1(0)\\v_1^\prime(0)\end{pmatrix}=M_3M_2M_1 \begin{pmatrix}v_1(0)\\v_1^\prime(0)\end{pmatrix}
\]
where 
\[M_j=\begin{pmatrix}\cos\gamma(\beta_j-\beta_{j-1})&\gamma^{-1}\sin\gamma(\beta_j-\beta_{j-1})\\
-\frac{\sigma_1}{\sigma_2}\gamma \sin\gamma(\beta_j-\beta_{j-1})
&\frac{\sigma_1}{\sigma_2}\cos\gamma(\beta_j-\beta_{j-1})\end{pmatrix}.\]

Hence the eigenvalue problem is equivalent to
\[det\left(M_3M_2M_1-I\right)=0.\]
The determinant above can be explicitly evaluated and has the form
\begin{align*}
&det\left(M_3M_2M_1-I\right)=\\
&2(1-\cos2\pi\gamma)+\mu_2\sin\gamma\beta_1\sin\gamma(2\pi-\beta_2)\cos\gamma(\beta_2-\beta_1)+\\
&+\mu_1\sin\gamma(\beta_2-\beta_1)\sin\gamma(2\pi-\beta_2)\cos\gamma\beta_1+\mu_3\sin\gamma\beta_1\sin\gamma(\beta_2-\beta_1)\cos\gamma(2\pi-\beta_2),
\end{align*}
where
\[\mu_2=\frac{\sigma_3}{\sigma_1}+\frac{\sigma_1}{\sigma_3}-2,\,\,\mu_1=\frac{\sigma_3}{\sigma_2}+\frac{\sigma_2}{\sigma_3}-2,\,\,\mu_3=\frac{\sigma_2}{\sigma_1}+\frac{\sigma_1}{\sigma_2}-2.\]
Note that the coefficients $\mu_j$ are non negative and $1-\cos2\pi\gamma>0$ for $\gamma\in(0,1)$, hence,
\[det\left(M_3M_2M_1-I\right)>0\]
for 
\[0<\gamma\leq \frac{1}{2}\min\left\{\frac{\pi}{\beta_2-\beta_1},\frac{\pi}{\beta_1},\frac{\pi}{2\pi-\beta_2}\right\}.\]
Since, by assumption \eqref{ipangoli}
\[\frac{1}{2}\min\left\{\frac{\pi}{\beta_2-\beta_1},\frac{\pi}{\beta_1},\frac{\pi}{2\pi-\beta_2}\right\}\geq \frac{1}{2}\frac{\pi}{\pi-\overline{\beta}}>\frac{1}{2},\]
we have that the first non zero eigenvalue $\gamma_1$ is strictly larger than $\frac{1}{2}$.
 \end{proof}

\textit{Proof of Theorem \ref{stimagrad}.} 
Let us consider the series expansion \eqref{serieaut} where $v^{(j)}$ are eingenfuctions related to eigenvalue $\gamma_j$ with
\begin{equation}\label{normalizz}\int_{0}^{2\pi} a(v^{(j)})^2d\theta=1.\end{equation}
The weak form of equation \eqref{eqaut} gives
\begin{equation*}
	\int_{0}^{2\pi}\left(a\frac{\der v^{(j)}}{\der\theta} \frac{\der w}{\der\theta}  -a\gamma_j^2v^{(j)}w\right)d\theta=0\mbox{ for every }w\in H^1_a(S^1).
\end{equation*}
By choosing $w=v^{(j)}$ we have, by \eqref{normalizz},
\begin{equation*}
	\int_{0}^{2\pi}a\left(\frac{\der v^{(j)}}{\der\theta}\right)^2d\theta=\gamma_j^2\int_{S^1}a(v^{(j)})^2d\theta=\gamma_j^2.
\end{equation*}
Now we recall (see \cite{Brezis}) that for some universal constant $c$ 
\[|v^{(j)}(\theta)|\leq c\|v^{(j)}\|_{H^1(S^1)},\]
and, hence, since $\gamma_j>1/2$, there is a constant $C$ depending only on $\sigma_0$ such that
\begin{equation}\label{2-11}
	\left|v^{(j)}(\theta)\right|\leq C\gamma_j \mbox{ for }0\leq\theta\leq 2\pi.
\end{equation}
From \eqref{serieaut} and \eqref{2-11}, by H\"older inequality and by \eqref{Kappa}, we have for $0<r\leq\frac{r_0}{2}$
\begin{eqnarray}\label{ur}
	|u_r(r,\theta)|&\leq& Cr^{\gamma_1-1}\sum_{j=1}^{\infty}|C_j|r^{\gamma_j-\gamma_1}\gamma_j^2\nonumber\\
	&\leq&\frac{C}{r_0}\left(\frac{r}{r_0}\right)^{\gamma_1-1}\left(\sum_{j=1}^\infty\left(\frac{r}{r_0}\right)^{2\gamma_j}\gamma_j^2\right)^{1/2}\left(\sum_{j=1}^{\infty}C_j^2r_0^{2\gamma_j}\gamma_j^2\right)^{1/2}\nonumber\\
	&\leq&\frac{C\sqrt{\tilde{C}K}}{r_0}\left(\frac{r}{r_0}\right)^{\gamma_1-1}
\end{eqnarray}
where $\tilde{C}=\sum_{j=1}^\infty 2^{-2\gamma_j}\gamma_j^2$ (the convergence of this series is a consequence of \eqref{asintautoval}).
Moreover, by equation \eqref{eqaut} we get that
\begin{equation}
	\left(\frac{\der^2 v^{(j)}}{\der\theta^2}\right)(\theta)=\gamma_j^2v^{(j)}(\theta) \mbox{  in }(0,2\pi)\setminus\{\beta_1,\beta_2\},
\end{equation}
and, by \eqref{2-11}, we get 
\begin{equation}
	\left|\left(\frac{\der^2 v^{(j)}}{\der\theta^2}\right)(\theta)\right|\leq C\gamma_j^3  \mbox{  in }(0,2\pi)\setminus\{\beta_1,\beta_2\}.
\end{equation}
By \eqref{2-11}, (2.14), Sobolev Imbedding Theorem and interpolation inequalities in each subset of $[0,2\pi]$ in which $a$ is constant, we have
\begin{equation}\label{1-14}
	\left\|\left(\frac{\der v^{(j)}}{\der\theta}\right)\right\|_{L^\infty([0,2\pi])}\leq \frac{C}{\overline{\beta}}\gamma_j^2,
\end{equation}
where $C$ depends on $\sigma_0$.
Then, proceeding as before,
\begin{equation}\label{utheta}
	\frac{1}{r}|u_\theta(r,\theta)|\leq  \frac{C}{\overline{\beta}r_0}\left(\frac{r}{r_0}\right)^{\gamma_1-1}\sqrt{\tilde{C}K}.
\end{equation}
From \eqref{ur} and \eqref{utheta}, for $0<r<\frac{r_0}{2}$, we have
\begin{equation}
	|\nabla u|\leq \frac{C}{\overline{\beta}r_0}\left(\frac{r}{r_0}\right)^{\gamma_1-1}\sqrt{\tilde{C}K}
\end{equation}
on each $D_k$ for $k=1,2,3$. 
By  \eqref{Kappa},  $\sqrt{K}$ can be bounded in terms of $\|u\|_{H^1(B)}$.
\qed
\begin{rem}
Estimate \eqref{gradest} holds true also if coefficient $\sigma$ attains only two different values on two non degenerate sectors, see \cite{BFI}.

Nevertheless, if we consider a vertex at which more than 3 sides intersects, then the estimate is not true anymore. A counterexample of this estimate can be easily constructed in the case of four equal sectors. See \cite[Lemma 1]{PS}. 

Moreover, if assumption \eqref{ipangoli} is not satisfied, the first positive eigenvalue can be smaller than $1/2$: for example, if $\beta_1=\pi/6$, $\beta_2=\pi/3$, $\sigma_1=10^{-1}$, $\sigma_2=10^3$ and $\sigma_3=10$, direct calculation shows that, for $\gamma=1/2$,  $det\left(M_3M_2M_1-I\right)<0$, hence the first positive eigenvalue is smaller that $1/2$.

\end{rem}
\section{Shape derivative of the solution of a Neumann problem with respect to movements of a polygonal partition}\label{shape}
Let $\om\subset\RR^2$ be a bounded open set such that $\der\om$ is Lipschitz continuous with constants $r_0$ and $K_0$ and $diam(\om)\leq L$.

Let us consider a polygonal partition $\mathcal{P}\subset\om$ such that $dist(\mathcal{P},\der\om)\geq d_0$ and such that
\begin{equation*}
	\mathcal{P}=\cup_{i=1}^M\overline{\mathcal{P}}_i,
\end{equation*}
where $\mathcal{P}_i$ is an open polygon.

Let us denote by $Q_1,\ldots,Q_{N}$ the vertexes of the polygons that compose $\mathcal{P}$.

Let us also assume that:
\begin{equation*}
\mbox{each }Q_j\mbox{ does not belong to more the three sides of polygons};
\end{equation*}
\begin{equation*}
dist(Q_j,Q_k)\geq d_0 \mbox{ if }j\neq k;
\end{equation*}
\begin{equation*}
\mbox{each polygon }\mathcal{P}_i\mbox{ contains a disk of radius greater than } r_1
\end{equation*}
denoting by $\beta_j^k$, $k=1,\ldots,k_j\leq 3$, the angles in the vertex $Q_j$, we assume there exists $\overline{\beta}\in (0,\pi)$ 
such that
\begin{equation}\label{ippart}
\begin{split}
\mbox{if }k_j=2,\quad &0<\overline{\beta}<\beta_j^k<2\pi-\overline{\beta}\mbox{ for }k=1,2\\ 
\mbox{if }k_j=3,\quad &0<\overline{\beta}<\beta_j^k<\pi-\overline{\beta}\mbox{ for }k=1,2,3.
\end{split}\end{equation}
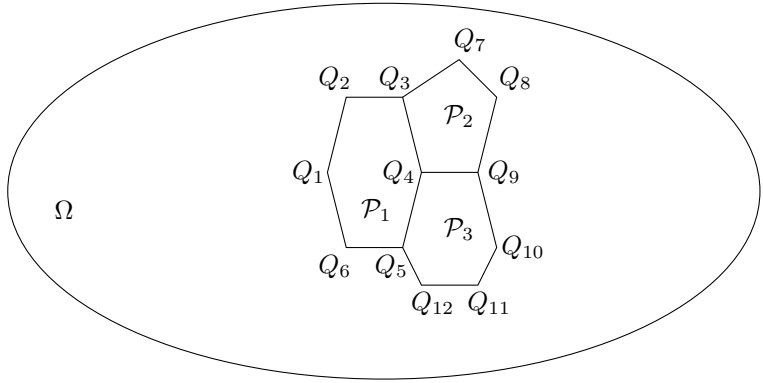
\begin{figure}
\centering
\begin{tikzpicture}[scale = 0.5]
                \draw  (5,5) ellipse [x radius = 10cm, y radius = 5cm];
								\draw  (3.5,5.5) -- (4,7.5) -- (5.5,7.5) -- (6,5.5) -- (5.5,3.5) -- (4,3.5) -- (3.5,5.5);
								\draw (6,5.5) -- (7.5,5.5) -- (8,3.5) -- (7.5,2.5) -- (6,2.5) -- (5.5,3.5);
								\draw (5.5,7.5) -- (7,8.5) -- (8, 7.5) -- (7.5,5.5);
								
								\draw (-3.5,4.5) node { $\Omega$};
								\draw (4.8,4.5) node {$\mathcal{P}_1$};
								\draw (7,7) node {$\mathcal{P}_2$};
								\draw (7,4) node {$\mathcal{P}_3$};
								\draw (3,5.5) node {$Q_1$};
								\draw (3.7,8) node {$Q_2$};
								\draw (5.3,8) node {$Q_3$};
								\draw (7.3,9) node {$Q_7$};
								\draw (8.5,8) node {$Q_8$};
								\draw (8.2,5.5) node {$Q_9$};
								\draw (5.4,5.5) node {$Q_4$};
								\draw (5.2,3) node {$Q_5$};
								\draw (3.7,3) node {$Q_6$};
								\draw (8.7,3.5) node {$Q_{10}$};
								\draw (7.8,2) node {$Q_{11}$};
								\draw (6.3,2) node {$Q_{12}$};
    \end{tikzpicture}	
\caption{Example of admissible polygonal partition.} \label{fig1}
\end{figure}
Let 
\begin{equation*}
	\sigma_0(x)=\sum_{i=1}^M\sigma_i(x)\chi_{\mathcal{P}_i}+\sigma_{M+1}\chi_{\om\setminus\mathcal{P}},
\end{equation*}
with
\begin{equation*}
	0<c_0^{-1}<\sigma_i<c_0,\mbox{ for every }i=1,\ldots,M+1.
\end{equation*}
We will sometimes use the notation $\mathcal{P}_{M+1}=\om\setminus\mathcal{P}$.

Let $f\in H^{-1/2}(\der\om)$ such that $\int_{\der\om}f=0$ and let $u_0\in H^1(\om)$ be the unique solution to
the boundary value problem
\begin{equation*}
    \left\{\begin{array}{rcl}
             \dive(\sigma_0\nabla u_0) & = & 0\mbox{ in }\om, \\
             \sigma_0\frac{\der u_0}{\der\nu} & = & f \mbox{ on }\der\om,\\
						\int_{\der\om}u_0&=&0,
           \end{array}
    \right.
\end{equation*}
where $\nu$ denotes the unit outer normal to $\der\om$.

Let $V=(v_1,\ldots,v_{N})\in \RR^{2N}$ be an arbitrary vector that represents the movements of vertexes of the polygons. 

For $t\geq 0$ let $\Psi^V$ be a function defined on $\cup_{i=1}^M\der\mathcal{P}_i$, such that, if $\overline{Q_jQ_k}$ is a side of one of the polygons, we have
\[\Psi^V(x):= v_j+\frac{(x-Q_j)\cdot(Q_k-Q_j)}{|Q_k-Q_j|}(v_k-v_j)\mbox{ for }x\in \overline{Q_jQ_k}.\]

We extend the function $\Psi^V$ to a $W^{1,\infty}$ function with compact support in $\om$.

Let $\Phi_t(x)=x+t\Psi^V(x)$, denote by $\mathcal{P}_i^t$ the polygon whose boundary is given by $\Phi_t(\der\mathcal{P}_i)$ and let
$\mathcal{P}^t=\cup_{i=1}^M\mathcal{P}^t_i$. The points $Q_j^t=Q_j+tv_j$ for $j=1,\ldots,N$ are the vertexes of polygons in $\mathcal{P}^t$.

For $t$ sufficiently small (depending on $V$, $r_1$, $\overline{\beta}$ and $d_0$) the new partition has the same properties of the original one, with slightly different constants. 

Let 
\begin{equation*}
	\sigma_t(x)=\sum_{i=1}^{M+1}\sigma_i(x)\chi_{\mathcal{P}^t_i}
\end{equation*}
and  let $u_t\in H^1(\om)$ be the unique solution to the boundary value problem
\begin{equation*}
    \left\{\begin{array}{rcl}
             \dive(\sigma_t\nabla u_t) & = & 0\mbox{ in }\om, \\
             \sigma_t\frac{\der u_t}{\der\nu} & = & f \mbox{ on }\der\om,\\
						\int_{\der\om}u_t&=&0.
           \end{array}
    \right.
\end{equation*}

The aim of this section is to evaluate, for $y\in\der\om$, the derivative of $u$ in the direction $V$, that is 
\[u^\prime(y)=\lim_{t\to 0} \frac{u_t(y)-u_0(y)}{t}.\]

As in \cite{BFV17}, thanks to Theorem \ref{stimagrad}, we can obtain this derivative by direct calculation, but, since the geometry of the problem makes these calculations quite involved, we follow here a different strategy. \\
Let $\tilde{u}_t(x)=u_t\circ\Phi_t(x)$  and let us evaluate the material derivative $\dot{u}$, that is the weak limit of $\frac{\tilde{u}_t-u}{t}$. 
Then, from the material derivative $\dot{u}$ we obtain the boundary values of the shape derivative $u^\prime$.

Note that for sufficiently small $t$ ( $t\leq \frac{1}{2\|\Psi^V\|_{W^{1,\infty}}}$) the function $\Phi^{-1}_t$ exists in $\Omega$. Let us define
\begin{equation}\label{2.6}
	A(t)=\left(D\Phi_t^{-1}\right)\left(D\Phi_t^{-1}\right)^T det\left(D\Phi_t\right)
\end{equation}
and
\begin{equation}\label{2.7}
	\mathcal{A}=\frac{dA}{dt}_{|_{t=0}}=div(\Psi^V)Id-(D\Psi^V+(D\Psi^V)^T)
\end{equation}
where $D\Phi_t^{-1}$ and $D\Psi^V$ represent the Jacobian matrices of $\Phi_t^{-1}$ and $\Psi^V$.\\

Let $\tilde{u}_t(x)=u_t\circ\Phi_t(x)$  and let us evaluate the material derivative $\dot{u}$, that is the weak limit of $\ds{\frac{\tilde{u}_t-u}{t}}$. 
\begin{lm}\label{lemmamaterial}
The material derivative $\dot{u}\in H^1(\om)$ solves 
\begin{equation}\label{material}
	\int_{\om}\sigma_0\nabla \dot{u}\cdot\nabla w=-\int_\om \sigma_0 \mathcal{A}\nabla u\cdot\nabla w \quad\forall w\in H^1(\om) 
\end{equation}
with $\int_{\der\om}\dot{u}=0$.
\end{lm}
See \cite{A} and \cite{BMPS} for the proof.\\

We now want to write equation \eqref{material} in a different way by integration by parts. Since the functions involved are not regular enough to perform this integration, we need to analyze carefully  what happens close to vertexes. This is the point where Theorem \ref{stimagrad}  comes into play. 

\begin{prop}\label{proposizione}
Let us denote by $S_k$ for $k=1,\ldots,M_1$ the sides of the polygons in $\mathcal{P}$.
For each  $v\in H^1(\om)$ solution of
\[\dive\left(\sigma_0\nabla v\right)=0\mbox{ in }\om,\]
we have,
\begin{equation}\label{42}
	\int_\om\sigma_0\nabla\dot{u}\cdot\nabla v=\sum_{k=1}^{M_1}\int_{S_k}[\sigma_0b]\cdot n_kds,
\end{equation}
 where 
\begin{equation}\label{b}
	b=\left(\Psi^V\cdot\nabla u_0\right)\nabla v+\left(\Psi^V\cdot\nabla v\right)\nabla u_0-\left(\nabla u_0\cdot\nabla v\right)\Psi^V,
\end{equation}
$n_k$ is a normal unit vector to $S_k$ and $[\sigma_0b]=\sigma^-b^--\sigma^+b^+$ where $\sigma^-,b^-$ are the functions $\sigma_0$, $b$  restricted to the polygon with side $S_k$ and with outer normal $n_k$ while $\sigma^+,b^+$ are the functions $\sigma_0$, $b$  restricted to the polygon with side $S_k$ and with inner normal $n_k$.
 \end{prop}
\begin{proof}
For $0<\ep<\frac{d_0}{4}$, let 
\[B_\ep=\cup_{j=1}^NB(P_j,\ep),\]
and let us denote by
\[u_i={u_0}_{|_{\mathcal{P}_i}}\mbox{ and }v_i={v}_{|_{\mathcal{P}_i}} \mbox{ for }i=1,\ldots,M.\]
Each of these functions is harmonic in $\mathcal{P}_i$; moreover $u_i, v_i\in H^2(\mathcal{P}_i\setminus B_\ep)$ and, by the regularity estimates in \cite{LN}, $u_i, v_i\in W^{1,\infty}(\overline{\mathcal{P}_i\setminus B_\ep})$.

The functions $u_{M+1}, v_{M+1}$ are harmonic in $\mathcal{P}_{M+1}$ and belong to $H^2(\mathcal{P}_{M+1}\setminus B_\ep)$ and to  
$W^{1,\infty}\left(\overline{\mathcal{P}_j\setminus \left(B_\ep\cap \{x\in\om\,:\,d(x,\der\om)<\ep\}\right)}\right)$.
Let us now consider equation \eqref{material} with $w=v$ and write
\begin{equation}\label{mat2}
	\int_\om \sigma_0\nabla\dot{u}\cdot\nabla v=-\int_{\om\setminus B_\ep} \sigma_0\mathcal{A}\nabla u_0\cdot\nabla v-
	\int_{B_\ep} \sigma_0\mathcal{A}\nabla u_0\cdot\nabla v.
\end{equation}
 In each set $\mathcal{P}_i\setminus B_\ep$ we have that
\begin{equation}\label{3.10}
	-\mathcal{A}\nabla u_0\cdot\nabla v=\dive(b)
\end{equation}
for $b$ given by \eqref{b}. Here we also used the fact that $\Delta u_i=\Delta v_i=0$ in $\mathcal{P}_i$.

Now, we integrate by parts in each $\mathcal{P}_j\setminus B_\ep$ and, recalling that $\Psi^V$ and, hence, $b$ have compact support in $\om$, we have
\begin{equation}\label{114}
	-\int_{\om\setminus B_\ep} \sigma_0\mathcal{A}\nabla u_0\cdot\nabla v=\sum_{i=1}^{M+1} \int_{\mathcal{P}_i}\sigma_i \dive(b)
=\sum_{k=1}^{M_1}\int_{S_k\setminus B_\ep}[\sigma_0b]\cdot n_k+\int_{\der B_\ep} \sigma_0 b\cdot n,
\end{equation}
where $n$ is the exterior normal to $\der B_\ep$.
 By putting together \eqref{mat2} and \eqref{114} we have
\begin{equation}\label{22}
	\int_\om \sigma_0\nabla\dot{u}\cdot\nabla v=\sum_{k=1}^{M_1}\int_{S_k\setminus B_\ep}[\sigma_0b]\cdot n_k+\int_{\der B_\ep} \sigma_0 b\cdot n-\int_{B_\ep} \sigma_0\mathcal{A}\nabla u_0\cdot\nabla v.
\end{equation}
Functions $u_0$ and $v$ both solve the same equation and, hence, for the assumption \eqref{ippart} on the polygons, they satisfy estimate  \eqref{gradest}. Then, we have
\begin{equation}\label{1-2}
	\left|\int_{B_\ep} \sigma_0\mathcal{A}\nabla u_0\cdot\nabla v\right|\leq C\ep^{2\gamma}
\end{equation}
and
\begin{equation}\label{1-3}
	\left|\int_{\der B_\ep} \sigma_0 b\cdot n\right|\leq C\ep^{2\gamma-1}.
\end{equation}
Since $\gamma>1/2$ (see Theorem \ref{stimagrad}), both the integrals in the right hand side of \eqref{22} tend to zero for $\ep\to 0$. Moreover, again by \eqref{gradest}, for $\ep\to 0$
\begin{equation}\label{1-4}
 \int_{S_k\setminus B_\ep}[\sigma_0b]\cdot n_k\to \int_{S_k}[\sigma_0b]\cdot n_k.
\end{equation}
By \eqref{22}, \eqref{1-2}, \eqref{1-3} and \eqref{1-4} we have \eqref{42}.
\end{proof}

\begin{rem}\label{calcolosalto}
Let us evaluate more precisely the jump $[\sigma_0b]$. 
 
Denoting by $\tau_k$ a direction orthogonal to $n_k$ we have, 
\begin{eqnarray}\label{salto}
	[\sigma_0b]\cdot n_k&=&\left[\sigma_0\left(\Psi^V\cdot\nabla u_0\right)\frac{\der v}{\der n_k}+\sigma_0\left(\Psi^V\cdot\nabla u_0\right)\frac{\der u_0}{\der n_k}\right.\nonumber\\&&\left.\phantom{\frac{\der u_0}{\der n_k}}- \sigma_0(\nabla u_0\cdot\nabla v)\Psi^V\cdot n_k\right]\nonumber\\
	&=&(\Psi^V\cdot n_k)\left[\sigma_0 \frac{\der u_0}{\der n_k}\frac{\der v}{\der n_k}-\sigma_0\frac{\der u_0}{\der \tau_k}\frac{\der v}{\der \tau_k} \right]\nonumber\\&&+(\Psi^V\cdot \tau_k)\left[\sigma_0 \frac{\der u_0}{\der \tau_k}\frac{\der v}{\der n_k}+\sigma_0\frac{\der u_0}{\der n_k}\frac{\der v}{\der \tau_k}\right].
\end{eqnarray}
By transmission conditions across $S_k$ for solution of the equation $\dive\left(\sigma_0\nabla u\right)$, we have
  \begin{equation}
\label{tange}
\left[\sigma_0 \frac{\der u_0}{\der \tau_k}\frac{\der v}{\der n_k}+\sigma_0\frac{\der u_0}{\der n_k}\frac{\der v}{\der \tau_k}\right]=0	 
 \end{equation}
and
  \begin{equation}
\label{norma}
\left[\sigma_0 \frac{\der u_0}{\der n_k}\frac{\der v}{\der n_k}-\sigma_0\frac{\der u_0}{\der \tau_k}\frac{\der v}{\der \tau_k} \right]
=(\sigma^--\sigma^+)\left(\frac{\sigma^+}{\sigma^-}\frac{\der u^+}{\der n_k}\frac{\der v^+}{\der n_k}+\frac{\der u^+}{\der \tau_k}\frac{\der v^+}{\der \tau_k}\right).
 \end{equation}
\end{rem}

\subsection{Boundary values of the shape derivative}
We now want to obtain the boundary values of the shape derivatives $u^\prime$. 
Since, by chain rule, 
\[u^\prime=\dot{u}-\Psi^V\cdot\nabla u\]
and $\Psi^V$ has compact support in $\om$, it is enough to get the boundary values of $\dot{u}$.
Let us now consider the Neumann function $N$ with pole at the boundary of $\om$, that is, for $y\in\der\om$ the unique solution to the boundary value problem
\begin{equation*}
    \left\{\begin{array}{rcl}
             \dive(\sigma_0\nabla N(\cdot,y) )& = & 0\mbox{ in }\om, \\
             \sigma_0\frac{\der N}{\der\nu}(\cdot,y) & = & -\delta_y(\cdot)+\frac{1}{|\der\om|} \mbox{ on }\der\om,
           \end{array}
    \right.
\end{equation*}
Let $y$ be a fixed point on $\der\om$. It is well known that $N(\cdot,y)$ is in $W^{1,1}(\om)$. Then, since  $\Psi^V$ has compact support in $\om$ and $\mathcal{P}\subset\om_{d_0}$,  it is possible to construct a sequence $v_m\in C^1(\om)$ that converges to $N(\cdot,y)$ in $W^{1,1}(\om)$ and in $C^1(\om_{d_0})$. Moreover since  $\dot{u}$ is smooth near $\der\Omega$ we can insert $v_m$ into \eqref{42} and pass to the limit, concluding that
\begin{equation}\label{boundary}
	u^\prime(y)=\dot{u}(y)=\sum_{k=1}^{M_1}\int_{S_k}(\sigma^--\sigma^+)\left(\frac{\sigma^+}{\sigma^-}u^+_{n_k}N^+_{n_k}(y,\cdot)+
	u^+_{\tau_k} N^+_{\tau_k}(y,\cdot)\right)(\Psi^V\cdot n_k)ds,
\end{equation}
which is the same formula we have in \cite[Theorem 4.6]{BFV17} for $g=-\delta_y+\frac{1}{|\der\om|}$.\\
\begin{rem}
The  Neumann-to-Dirichlet  map is the operator $\mathcal{N}_{\sigma_0}: H_0^{-1/2}(\partial \Omega) \to H_0^{1/2}(\partial \Omega)$, defined by
\begin{equation}
\mathcal{N}_{\sigma_0}(f) = u|_{\partial \Omega},
\end{equation}
where $H^{s}_0(\partial \Omega) = \{ f \in H^s(\partial \Omega) : \int_{\partial \Omega} f = 0\}$, $g \in H_0^{-1/2}(\partial \Omega)$ and $u$ is the unique $H^1(\Omega)$ weak solution of the Dirichlet problem for the conductivity equation
\begin{equation}
\nabla \cdot (\sigma_0 \nabla u) = 0 \; \textrm{on} \; \Omega, \; \; \; \left. \sigma_0 \frac{\partial u}{\partial n}\right|_{\partial \Omega}=f,
\label{schr}
\end{equation}
satisfying the normalization condition
$$\int_{\partial \Omega} u\, d\sigma = 0,
$$ 
where $\nu$ is the outer normal of $\partial \Omega$.

Let $\mathcal{P}$ denote a partition of vertices $Q=(Q_1,Q_2,\dots,Q_N)$ and denote by $\mathcal{Q}$ the subset of points $Q\in\Omega_{d_0}^N$ satisfying the assumptions stated at the beginning of Section 3.  For $f,g \in H^{-1/2}_{0}(\partial\Omega)$ we can define $\tilde{F}:\mathcal{Q}\rightarrow \RR$ as follows
\[
\tilde{F}(Q)=<g, \mathcal{N}_{{\sigma_0}}(f)>\quad\forall Q\in \mathcal{Q}.
\]
Let $Q^t=Q+tV$. Then 
\[
\frac{d\tilde{F}(Q^t)}{dt}|_{t=0}=\int_{\partial\Omega}gu'.
\]
Now, observing that $u'=\dot{u}$ on $\partial\Omega$ 
\[
\int_{\partial\Omega}gu'=\int_{\partial\Omega}g\dot{u}=\int_{\Omega}\sigma_0\nabla w\cdot\nabla\dot{u}
\]
where $w\in H^1(\Omega)$ solves
\begin{equation}
\nabla \cdot (\sigma_0 \nabla w) = 0 \; \textrm{on} \; \Omega, \; \; \; \left. \sigma_0 \frac{\partial w}{\partial n}\right|_{\partial \Omega}=g
\end{equation}
and from Proposition \ref{proposizione} we get that
\[
\frac{d\tilde{F}(Q^t)}{dt}|_{t=0}=\sum_{k=1}^{M_1}\int_{S_k}(\sigma^--\sigma^+)\left(\frac{\sigma^+}{\sigma^-}u^+_{n_k}w^+_{n_k}(y,\cdot)+
	u^+_{\tau_k} w^+_{\tau_k}(y,\cdot)\right)(\Psi^V\cdot n_k)ds.
\]
Finally, arguing similarly as in \cite{BFV17} it is possible to establish that also $F$ is differentiable.

\end{rem}

\begin{rem}
Proposition \ref{proposizione} holds true also in different assumptions on the geometry of the domain. For example if there are more than one polygons $\mathcal{P}$ inside the domain (see Figure \ref{fig2} on the left) or if the polygons are nested (see Figure \ref{fig2} on the right). The only condition on the partition is that each vertex has positive distance from the boundary of $\om$ and from the other vertexes and that there are no more that 3 sides intersecting at each vertex.   
\end{rem}
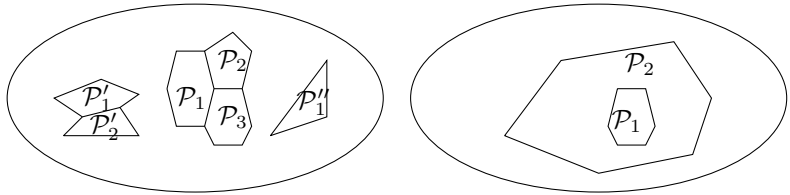
\begin{figure}
\centering
\begin{tabular}{cc}
\begin{tikzpicture}[scale = 0.25]
                \draw  (5,5) ellipse [x radius = 10cm, y radius = 5cm];
								\draw  (3.5,5.5) -- (4,7.5) -- (5.5,7.5) -- (6,5.5) -- (5.5,3.5) -- (4,3.5) -- (3.5,5.5);
								\draw (6,5.5) -- (7.5,5.5) -- (8,3.5) -- (7.5,2.5) -- (6,2.5) -- (5.5,3.5);
								\draw (5.5,7.5) -- (7,8.5) -- (8, 7.5) -- (7.5,5.5);
								\draw (-2,3) -- (-1,4) -- (1,4.5) -- (2,3)-- (-2,3);
								\draw (-1,4) -- (-2.5,5) -- (0,6) -- (2,5.2) -- (1,4.5);
								\draw (9,3) -- (12,7) -- (12,4) -- (9,3);
								\draw (4.8,5) node {$\mathcal{P}_1$};
								\draw (7,7) node {$\mathcal{P}_2$};
								\draw (7,4) node {$\mathcal{P}_3$};
								\draw (-0.2,5) node {$\mathcal{P}^\prime_1$};
								\draw (0.2,3.5) node {$\mathcal{P}^\prime_2$};
								\draw (11.3,4.8) node {$\mathcal{P}^{\prime\prime}_1$};
										\end{tikzpicture}

&
    \begin{tikzpicture}[scale = 0.25]
                \draw  (5,5) ellipse [x radius = 10cm, y radius = 5cm];
								\draw  (0,3) -- (5,1) -- (10,2) -- (11,5) -- (9,8) -- (3,7) -- (0,3);
								\draw (6,5.5) -- (7.5,5.5) -- (8,3.5) -- (7.5,2.5) -- (6,2.5) -- (5.5,3.5) -- (6,5.5);
								\draw (6.5,3.8) node {$\mathcal{P}_1$};
								\draw (7.2,6.8) node {$\mathcal{P}_2$};
    \end{tikzpicture}
\end{tabular}

\caption{Left: disjoint polygonal partitions; right: nested polygons} \label{fig2}
\end{figure}

\end{document}